\documentclass{article}
\usepackage[utf8]{inputenc}
\usepackage{amssymb}
\usepackage{amsmath, amssymb, amsthm}
\usepackage{url}
\usepackage{dsfont}
\usepackage{textcomp}
\usepackage{setspace}
\usepackage{float}
\usepackage{enumitem} % necessário para personalizar listas
\usepackage{mathpazo}       %use hott font
\usepackage{caption}
\usepackage[english]{babel}
\usepackage{graphicx}
\graphicspath{{./images/}}
\usepackage{amsmath}
\usepackage[dvipsnames]{xcolor}
\usepackage{booktabs}
\usepackage{rotating} % Required for the sidewaystable environment
\usepackage{geometry} % Allows for page margin adjustments

\usepackage[
    colorlinks=true,    % ativa cores em vez de caixas coloridas
    linkcolor=blue,     % cor de links internos (seções, figuras...)
    citecolor=green, % cor de citações bibliográficas
    filecolor=Magenta,  % cor de links para arquivos locais
    urlcolor=black   % cor de URLs
]{hyperref}
\newtheorem{theorem}{\sc Theorem}[section]

\newtheorem{lem}[theorem]{\sc Lemma}
\newtheorem{prop}[theorem]{\sc Proposition}
\newtheorem{cor}[theorem]{\sc Corollary}

\newtheorem{ex}[theorem]{\sc Example}

\usepackage[nottoc]{tocbibind}
\usepackage{cite}  % Para usar citações no formato numérico ou autor-ano

\begin{document}

\title{A note on connecting multiplicative functions defined on finite groups}
\author{João Victor Monteiros de Andrade\thanks{Department of Computing, University of Brasilia. \texttt{jotandrade98@gmail.com, andrade.monteiros@aluno.unb.br}} 
\and Leonardo Santos da Cruz\thanks{Department of Computing, University of Brasilia. \texttt{leonardo-7238@hotmail.com, santos-cruz.sc@aluno.unb.br}}}

\date{} % Descomente esta linha caso queira adicionar uma data
\maketitle

\section{Abstract}

In the present work we explore relationships between multiplicative functions defined in \cite{Garonzi_2018}, \cite{Lazorec}. To do so, we use other important quotients defined and studied in \cite{ndeg}, \cite{Marius}, thus establishing deeper connections that were previously little or not evaluated in some group families. Furthermore, we extract new characterizations and asymptotic patterns for some classes of groups.

\noindent{\textbf{Keywords:} Multiplicative functions; p-groups; GAP}

\section{Introduction}

\quad \, In \cite{Marius}, the cyclicity degree $(cdeg)$ is defined as a measure that evaluates the proportion between the number of cyclic subgroups and the total number of subgroups in the group. That work presented a series of interesting results for certain classes of finite groups, particularly subclasses of finite p-groups. In \cite{ndeg}, the degree of normality $(ndeg)$ was defined. This function, like the cyclicity degree, exhibits a series of interesting properties and is defined in a manner similar to the degree of simplicity, but instead of considering the number of cyclic subgroups, it considers the total number of normal subgroups in the group divided by the total number of subgroups. Analogous to the degree of simplicity, the degree of normality was applied to the same classes of finite and p-groups.

In \cite{Garonzi_2018}, the function $\alpha$ was defined, which enabled the establishment of various classification criteria for groups, applicable to several families of groups, such as non-solvable groups, among others. Finally, in \cite{Lazorec}, the function $\beta$ was introduced, which proved to be a very useful tool for evaluating the asymptotic behavior of certain families of groups, particularly some families of p-groups. Both the $\alpha$ and $\beta$ functions turned out to be multiplicative and also exhibited interesting characteristics with respect to their domain. In \cite{Lazorec}, the author proposes an investigation into the relationships between these functions (Problem 6.3), aiming to establish deeper connections beyond the order relation, which is clearly observed and also presented in the same work. In the present study, we present some of the use of direct relationships between some functions to verify the asymptotic behavior in some group families.

\paragraph{Notation.}
Throughout, $G$ denotes a finite group. Let $L(G)$ be the set of all subgroups of $G$, 
and let $c(G)$ be the number of cyclic subgroups of $G$. Define
\[
  \alpha(G) := \frac{c(G)}{|G|}, \qquad 
  \beta(G) := \frac{|L(G)|}{|G|}, \qquad 
  \mathrm{cdeg}(G) := \frac{c(G)}{|L(G)|} = \frac{\alpha(G)}{\beta(G)}.
\]
\quad\, We write $C_n$ for the cyclic group of order $n$, $C_p^n$ for the elementary abelian group $(C_p)^n$,
$D_{2n}$ for the dihedral group of order $2n$, and $D_{2^{n}}, Q_{2^{n}}, SD_{2^{n}}$ for the dihedral,
generalized quaternion, and semidihedral $2$-groups of order $2^{n}$, respectively.
The dicyclic group of order $4n$ is denoted by $\mathrm{Dic}_n$.

\section{Connections between functions}
\begin{theorem}\label{T1}

    Let $G$ be a finite group, then

\begin{equation}\label{eq. 1} 
    cdeg(G) = \frac{\alpha(G)}{\beta(G)}.
\end{equation}
\end{theorem}

\begin{proof}
    % \begin{equation*}
    %     cdeg(G) = \frac{c(G)}{|L(G)|} = \frac{c(G)}{|L(G)|} \frac{|G|}{|G|}  = \frac{c(G)}{|G|} \frac{|G|}{|L(G)|} = \alpha(G) \frac{1}{\beta(G)} = \frac{\alpha(G)}{\beta(G)}.
    % \end{equation*} 
    This result is immediate.
\end{proof}

This theorem establishes a strong relationship between the functions $\alpha$ and $\beta$ through the cyclicity degree. Another interesting connection between these multiplicative functions can be obtained for the specific class of groups with structure $C_{p} \rtimes C_{q^{n}}$. In \cite{deandrade2025}, a function $\mathfrak{J}$ is defined as the quotient of the number of nilpotent subgroups of a finite group $G$ by the total number of subgroups of the group. In that work, it was shown that $\mathfrak{J}(C_{p} \rtimes C_{q^{n}}) = cdeg(C_{p} \rtimes C_{q^{n}})$. Based on the equality established in \ref{T1}, it is possible to verify that in this particular case, $\mathfrak{J}(C_{p} \rtimes C_{q^{n}}) = \frac{\alpha(C_{p} \rtimes C_{q^{n}})}{\beta(C_{p} \rtimes C_{q^{n}})}$. 

\subsection{General results}

\begin{prop} \label{prop 3.4}
    For the groups in \{ $D_{2^n}, Q_{2^{n}},SD_{2^n}$ \} we have:
\begin{itemize}[label=$\circ$]
    \item{$\beta(D_{2^n}) = \frac{2^{n} + n -1}{2^n} ;$} 
    \item{$\beta(Q_{2^{n}}) = \frac{2^{n-1} + n-1}{2^n} ;$} 
    \item{$\beta(SD_{2^n}) = \frac{3\cdot2^{n-2} + n-1}{2^n}.$} 
\end{itemize}
   
\end{prop}

Where the group presentations are:

\begin{itemize}
    \item{$D_{2^n} = \bigl\langle\,x,y \bigm| x^{2^{n-1}} = 1,\; y^2 = 1,\; y\,x\,y = x^{-1} \bigr\rangle$;} 
    \item{$Q_{2^{n}} = \langle x, y \mid x^{2^{n-1}} = y^4 = 1, \; yxy^{-1} = x^{2^{n-1}-1}\rangle;$} 
    \item{$SD_{2^n} = \langle x, y \mid x^{2^{n-1}} = y^2 = 1,  y^{-1}xy = x^{2^{n-2}-1} \rangle \quad (n \geq 4).$}
      
\end{itemize}

\begin{proof}
 This result is immediate.   
\end{proof}

\begin{cor}

  \begin{itemize}[label=$\circ$]
Considering the limits in the expressions of the proposition \ref{prop 3.4} we have
  
    \item{$\lim_{n \longrightarrow \infty}\beta(D_{2^n}) = 1;$} 
    \item{$\lim_{n \longrightarrow \infty}\beta(Q_{2^{n}}) = \frac{1}{2};$} 
    \item{$\lim_{n \longrightarrow \infty}\beta(SD_{2^n}) = \frac{3}{4}.$} 
\end{itemize}
\end{cor}

% Note that the above groups are not abelian, but similar characterizations in terms of asymptotic results can be obtained for families of abelian groups, for example, $C_{2^{n}p}\times C_{2}$. 

Asymptotic results can also be obtained for $\alpha$, considering the equation \ref{eq. 1}.  

\begin{prop} \label{prop 3.9}
    For the groups in \{ $D_{2^n}, Q_{2^{n}},SD_{2^n}$ \} we have:
\begin{itemize}[label=$\circ$]
    \item{$\alpha(D_{2^n}) = \left(\frac{2^{n-1} +n }{2^n + n -1 } \right) \left( \frac{2^{n} + n -1}{2^n}  \right);$} 
    \item{$\alpha(Q_{2^{n}}) = \left(\frac{2^{n-2} + n}{2^{n-1} +n -1} \right) \left( \frac{2^{n-1} + n-1}{2^n} \right);$} 
    \item{$\alpha(SD_{2^n}) = \left( \frac{3 \cdot 2^{n-3} + n}{3 \cdot 2^{n-2} +n -1} \right) \left( \frac{3\cdot2^{n-2} + n-1}{2^n} \right).$} 
\end{itemize}
\end{prop}

\begin{proof}

This result follows immediately from Theorems 3.3.4, 3.3.6 and 3.3.8 of \cite{Marius}, as well as from Proposition \ref{prop 3.4}, when applied to the equation \ref{eq. 1}.
\end{proof}

\begin{cor}

  \begin{itemize}[label=$\circ$]
Considering the limits of the proposition \ref{prop 3.9} we have
  
    \item{$\lim_{n \longrightarrow \infty}\alpha(D_{2^n}) = \frac{1}{2};$} 
    \item{$\lim_{n \longrightarrow \infty}\alpha(Q_{2^{n}}) = \frac{1}{4};$} 
    \item{$\lim_{n \longrightarrow \infty}\alpha(SD_{2^n}) = \frac{3}{8}.$} 
\end{itemize}
\end{cor}

Due to property 2.5 in \cite{Garonzi_2018} it is possible to generalize the above results to direct products. Furthermore, based on Theorem \ref{T1} it is possible to obtain a closed expression for $\alpha$ as shown in the following proposition.

\begin{prop}
    Let $G = D_{2n} \times C_{2}^m$, then

    \begin{equation*}
        \alpha(D_{2n} \times C_{2}^m) = \frac{n + \tau(n)}{2n}.
    \end{equation*}

Furthermore,

    \begin{equation*}
       \lim_{n \longrightarrow \infty} \alpha(D_{2n} \times C_{2}^m) = \frac{1}{2}.
    \end{equation*}

\end{prop}

\begin{proof}
By Property 2.5 of \( \alpha \) in \cite{Garonzi_2018}, it follows that
\( \alpha(D_{2n} \times C_2^m) = \alpha(D_{2n}) \). By Corollary 3.4.2 in \cite{Marius}, we have

\begin{equation*}
        cdeg(D_{2n}) = \frac{n + \tau(n)}{\sigma(n) + \tau(n)}.
\end{equation*}

Then it follows

\begin{align*}
     \frac{n + \tau(n)}{\sigma(n) + \tau(n)}  &= \frac{2n\alpha(D_{2n})}{\sigma(n) + \tau(n)}\\
        n + \tau(n)  &=  2n\alpha(D_{2n})\\
        \alpha(D_{2n}) &= \frac{n + \tau(n)}{2n}.
\end{align*}

Now, consider

\begin{align*}
         \lim_{n \longrightarrow \infty}\frac{n + \tau(n)}{2n} =\lim_{n \longrightarrow \infty}\frac{n + \tau(n)}{2n} = \lim_{n \longrightarrow \infty}\frac{n }{2n} + \lim_{n \longrightarrow \infty}\frac{\tau(n)}{2n} = \frac{1}{2} + \lim_{n \longrightarrow \infty}\frac{\tau(n)}{2n}.
\end{align*}

Now let \( d \) be a divisor of \( n \). Then \( e = \frac{n}{d} \) is also a divisor of \( n \), and at least one between \( d \) and \( e \) satisfies \( d \le \sqrt{n} \). Therefore, the number of pairs \( (d,e) \) is at most \( \sqrt{n} \), and therefore:

\[
\tau(n) \le 2\sqrt{n}.
\]

It follows that

\[
0 \le \frac{\tau(n)}{2n}
\le \frac{2\sqrt{n}}{2n}
= \frac{1}{\sqrt{n}}
\xrightarrow[n \to \infty]{} 0.
\]

Therefore,

\[
\lim_{n \to \infty} \frac{\tau(n)}{2n} = 0.
\]
\quad \, Which concludes the demonstration.

\end{proof}

\section{Nilpotent Groups}
\subsection{Abelian case}

\begin{prop} \label{prop 3.6}
Let $G = C_{2^{n+1}} \times C_{2} $, then

\begin{equation*}
    \beta(C_{2^{n+1}} \times C_{2} ) = \frac{3n +5}{2^{n+2}}.
\end{equation*}
\end{prop}
    
\begin{proof}
Let $G = C_{2^{n+1}}\times C_{2}.$
By the Fundamental Theorem of Finitely Generated Abelian Groups, every subgroup 
$H\le G$ is also abelian and therefore isomorphic to
$ C_{2^r}\times C_{2^s}, \quad 0\le s\le1,\quad 0\le r\le n+1$, \cite{stehling1992computing}. Let us first consider cyclic subgroups \((s=0)\). In a cyclic group of order~\(2^{n+1}\), for each divisor \(2^r\) there is exactly one subgroup of order \(2^r\). Therefore, for \(r=0\), there is only the trivial subgroup, total \(1\); for \(r=1\) there are three involutions \(\,(2^n,0),\,(0,1)\) and \((2^n,1)\), therefore \(3\) subgroups of order 2. For \(2\le r\le n+1\) in \(G\) there are exactly two cyclic subgroups of order \(2^r\), as one comes from the only subgroup of order \(2^r\) in \(C_{2^{n+1}}\) with trivial factor in \(C_2\); the other comes from the only subgroup of order \(2^{r-1}\) in \(C_{2^{n+1}}\) added to the generator of \(C_2\). For \(r>n+1\) they do not exist, as the exponent of \(G\) is \(2^{n+1}\). Now let us consider the non‑cyclic subgroups \((s=1)\). For each \(1\le r\le n+1\), the direct product $ \langle 2^{\,n+1-r}\rangle \times C_2 \simeq C_{2^r}\times C_2$ is the only subgroup of \(G\) isomorphic to \(C_{2^r}\times C_2\). Finally, let \(N_k\) be the number of subgroups of order \(2^k\). Then there is

\[
N_k=
\begin{cases}
1, & k=0,\\
3, & k=1,\\
2, & 2\le k\le n+1,\\
0, & k>n+1.
\end{cases}
\]

For each \(1\le r\le n+1\), there is one more subgroup of order \(2^{r+1}\) (the non‑cyclic one), so that the total of subgroups is: $1+3+2n+(n+1)=3n+5.$ Dividing by the order of $C_{2^{n+1}}\times C_{2}$ follows the result.

\end{proof}

\begin{cor}
    Let $G = C_{2^{n+1}} \times C_{2} $, then $\lim_{n \longrightarrow \infty} \beta(C_{2^{n+1}} \times C_{2}) = 0$.
\end{cor}

\begin{prop}
   Let \( G = C_{2^{n}p} \times C_{2} \), where \( p \) is a prime with \( p \geq 3 \), then

    \begin{equation*}
        \beta(C_{2^{n}p}\times C_{2}) = \frac{3n +2}{2^{n}p}.
    \end{equation*}
\end{prop}

\begin{proof}

Let \( G = C_{2^n p} \times C_2 \), where \( p \) is an odd prime and \( n \geq 1 \). By the Fundamental Theorem of Finite Abelian Groups and the Chinese Remainder Theorem, we have:
\[
C_{2^n p} \cong C_{2^n} \times C_p \Rightarrow G \cong (C_{2^n} \times C_p) \times C_2 \cong (C_{2^n} \times C_2) \times C_p.
\]
\quad \, Since the primary decomposition of finite abelian groups is unique, it follows that the component associated with the 2-primary part of \( G \) is \( A = C_{2^n} \times C_2 \), and the component of order \( p \) is \( C_p \).

Let \( H \leq G \). Then there exists a unique decomposition \( H = H_2 \times H_p \), with \( H_2 \leq A \), \( H_p \leq C_p \). Therefore,
\[
 |L(G)|  = |L(A)| \cdot |L(C_p)|.
\]
\quad \, Since \( C_p \) is a cyclic group of prime order, it has exactly two subgroups: \( \{e\} \) and \( C_p \). Thus,
\[
|L(C_p)| = 2.
\]

For \( A = C_{2^n} \times C_2 \), we have the following cases:
\begin{itemize}
  \item \( n = 1 \): \( A \cong C_2 \times C_2 \Rightarrow  \) $|L(A)| = 5$.
  \item \( n = 2 \): \( A \cong C_4 \times C_2 \Rightarrow \) $|L(A)| = 8$.
  \item \( n = 3 \): \( A \cong C_8 \times C_2 \Rightarrow \) $|L(A)| =11$.
\end{itemize}

It follows that $|L(A)|$ grows as an arithmetic progression. That way,
$|L(A)| = |L(C_{2^n} \times C_2)| = 3n + 2$. Let us prove by induction on n. For the induction basis $|L(C_{2} \times C_2)| = 3(1) + 2 = 5$. Suppose it is valid for $ k\leq n$, thus $|L(C_{2^{k}} \times C_2)| = 3k + 2$. For $k+1$ it follows

\begin{align*}
    |L(C_{2^{k+1}} \times C_2)| = 3(k+1) + 2 = 3k+5,
\end{align*}

\noindent{which follows from Proposition \ref{prop 3.6}. Therefore, by, $|L(G)| = (3n + 2) \cdot 2 = 6n + 4.$ According to the definition of the function $\beta$, the result follows.}

\end{proof}

\begin{cor}
Let \( G = C_{2^n p} \times C_2 \), where \( p \) is a prime number with \( p \geq 3 \). Then

    \begin{equation*}
        \lim_{n \longrightarrow \infty}\beta(C_{2^{n}p}\times C_{2}) = 0.
    \end{equation*} 
\end{cor}

With the expressions that characterize asymptotic behaviors of group families, it is possible to practically immediately determine the degree of cyclicity. From the families of abelian groups above we have the result.

\begin{prop}
Let G be a finite group then
\begin{itemize}
        \item{ for $G = C_{2^{n+1}}\times C_{2}$, \quad $cdeg(C_{2^{n+1}}\times C_{2}) = \frac{2(n+2)}{3n+5};$}
        \item{ for $G = C_{2^{n}p}\times C_{2}$, \quad $cdeg( C_{2^{n}p}\times C_{2}) = \frac{2(n+1)}{3n+2}.$ }
    \end{itemize}
\end{prop}

\begin{proof}
$\alpha(C_{2^{n+1}}\times C_{2}) = \alpha(C_{2^{n+1}}) = \frac{n+2}{2^{n+1}}$, by the equation \ref{eq. 1} and by the Proposition
\ref{prop 3.6}, 

\begin{equation*}
  cdeg(C_{2^{n+1}}\times C_{2}) = \frac{n+2}{2^{n+1}} \cdot \frac{2^{n+2}}{3n+5} = \frac{2(n+2)}{3n+5}.  
\end{equation*}

$\alpha(C_{2^{n}p}\times C_{2}) = \alpha(C_{2^{n}} \times C_p) = \alpha(C_{2^{n}}) \cdot \alpha(C_p) = \frac{n+1}{2^{n}} \cdot \frac{2}{p}  = \frac{2(n+1)}{2^{n}p}$, in a similar way to the previous item, we have

\begin{equation*}
  cdeg(C_{2^{n}p}\times C_{2}) = \frac{2(n+1)}{2^{n}p} \cdot \frac{2^{n}p}{3n+2} = \frac{2(n+1)}{3n+2}.  
\end{equation*}

\end{proof}

\begin{cor}
\begin{equation*}
\lim_{n \longrightarrow  \infty}    cdeg(C_{2^{n+1}}\times C_{2}) = \lim_{n \longrightarrow  \infty}    cdeg(C_{2^{n}p}\times C_{2}) = \frac{2}{3}
\end{equation*}
\end{cor}

\subsection{Hamiltonian case}

\quad \, In \cite{Marius} Theorem 3.2.1 characterizes $cdeg(G)$ where $G$ is Hamiltonian, that is, $G \cong Q_8 \times C_2^{n} \times A$, where $C_2^{n}$ is an elementary abelian 2-group, and $A$ is a torsion abelian group with all elements of odd order. We will now evaluate the expressions of the other functions for this family of groups.

\begin{prop}
   Let $G \cong Q_8 \times C_2^{n} \times A$ be a Hamiltonian group, then $\alpha(G) = \frac{\alpha(A)5}{8}$.
\end{prop}

\begin{proof}
\begin{align*}
    \alpha(G) = \alpha(Q_8 \times C_2^{n} \times A) &= \alpha(Q_8 \times  C_2^{n}) \alpha(A) = \alpha(Q_8)\alpha(A) = \frac{\alpha(A)5}{8}
\end{align*}
\end{proof}

\begin{cor}
Let $G \cong Q_8 \times C_2^{n} \times A$ be a Hamiltonian group, then $\alpha(G) < \frac{15}{32}$.
\end{cor}

\begin{proof}
Since $A$ has odd order, then this group must not have involutions and therefore cannot be of the elementary abelian 2-group type, which are the only groups for which $\alpha(G)$ reaches the value of 1. According to Theorem 5 of \cite{Garonzi_2018}, if a nilpotent group has $\alpha(G)= \frac{3}{4}$, then it must be a 2-group. Since a finite abelian group of odd (non-trivial) order cannot be a 2-group, the equality to $\frac{3}{4}$ is impossible for this class of groups, thus establishing this value as a strict upper bound. Multiplying by $\alpha(Q_{8})$ follows the result.
 \end{proof}

\begin{cor}
Let $G \cong Q_8 \times C_2^{n} \times A$ be a Hamiltonian group such that $A = C_{p}^{m}$, for prime p, where $p \geq 3$, then

\begin{equation*}
    \lim_{m \longrightarrow \infty} \alpha( Q_8 \times  C_2^{n} \times C_{p}^{m}) = \frac{5}{8(p-1)}
\end{equation*}
\end{cor}

\begin{proof}
From proposition \ref{prop. 4.7} we have $\alpha(Q_8 \times C_2^{n} \times C_{p}^{m}) = \frac{5}{8}\alpha(C_{p}^{m})$. Taking the limit and applying Theorem 3 of \cite{Garonzi_2018} we have

\begin{equation*}
\lim_{m \longrightarrow \infty}\alpha( Q_8 \times  C_2^{n} \times C_{p}^{m}) = \frac{5}{8}\lim_{m \longrightarrow \infty}\alpha(C_{p}^{m}) = 
 \frac{5}{8} \frac{1}{\phi(\exp(C_{p}^{m}))} =  \frac{5}{8} \frac{1}{\phi(p)} =  \frac{5}{8(p-1)}    
\end{equation*}

\end{proof}

\begin{prop}\label{prop. 4.7}
  Let $G \cong Q_8 \times C_2^{n} \times A$ be a Hamiltonian group, then 

\begin{equation*}
    \beta(G) = \beta(A) \cdot\frac{b_{n,2}}{2^{n+3}}
\end{equation*}

\noindent{\text{where:} $
b_{n,2} = 2^{n+2} + 1 + 8 \sum_{\alpha=0}^{n-2} \left(2^{n-\alpha} - 2^{2\alpha+1} + 2^{\alpha}\right) a_{\alpha,2} + 2^{n+2} a_{n-1,2} + a_{n,2},$
\text{e} \( a_{\alpha,2} = |L(C_2^\alpha)| \), \text{for all} \(\alpha \in \mathbb{N}^*\).}
\end{prop}

\begin{proof}
Applying Theorem \ref{T1}, followed by proposition \ref{prop. 4.7} and Theorem 3.2.1 of \cite{Marius}, we have

\begin{align*}
   \beta(G) = \frac{\alpha(G)}{cdeg(G)} = \frac{\alpha(A)5}{8} \cdot \frac{b_{n,2}}{5 \cdot2^{n}\cdot cdeg(A)} = \frac{\alpha(A) \cdot b_{n,2}}{2^{n+3} \cdot cdeg(A)}
   &= \frac{c(A)}{|A|} \cdot b_{n,2} \cdot\frac{1}{2^{n+3}} \cdot \frac{|L(A)|}{c(A)}\\
   &= \frac{|L(A)|}{|A|} \cdot\frac{b_{n,2}}{2^{n+3}}\\
   &= \beta(A) \cdot\frac{b_{n,2}}{2^{n+3}}\\
 \end{align*}
\end{proof}

Note that for $ndeg(G) = 1$, when $G \cong Q_8 \times C_n^2 \times A$, since Hamiltonian groups have all normal subgroups.

\newpage
\section{Supersolvable Groups}
\subsection{Dicyclic case}

\quad \, Let us now evaluate the functions for some subfamilies of a class of groups that is not nilpotent.
This class is largely characterized by being described as a semidirect product, which sometimes makes the structural analysis of the groups more complex. This class, however,
has an interesting property regarding the quantitative description of its subgroups.

If G is dicyclic such that $|G| = 4n$, then $|L(G)| = \tau(2n) + \sigma(n)$. With this property,
it becomes very practical to determine the images of $\beta$. Let us study some subfamilies of dicyclic groups and verify the asymptotic behavior for $\beta$.

\begin{theorem}

Let $G$ be a finite dicyclic group where $|G| = 4n$, then

\begin{equation*}
        \beta(G) = \frac{\tau(2n) + \sigma(n)}{4n}.
\end{equation*}
\end{theorem}

\begin{proof}
It follows directly from the definition of $\beta$.
\end{proof}

\begin{prop} \label{prop 5.2}
If G is finite with p and q distinct primes, then, then

\begin{itemize}
    \item{$ \beta(C_{p^{n}} \rtimes Q_{2^{m}}) = \frac{m(n+1) + \left( \frac{p^{n+1}-1}{p-1}\right)(2^{m-1} -1 )}{2^{m} p^{n}}$, for $m \geq 3$};
    \item{$\beta (C_{q^m} \rtimes (C_{p^{n}} \rtimes C_4)) = \frac{2(n+1)(m+1) + \left( \frac{p^{n+1} - 1}{p - 1} \right) \left( \frac{q^{m+1} - 1}{q - 1} \right)}{4q^{m}p^{n}}$, for $p \neq q$}.
    \item{$\beta (C_{q^m} \rtimes (C_{p^{n}} \rtimes Q_{2^r})) = \frac{r(n + 1)(m + 1) + (2^{r-1} - 1) \left( \frac{p^{n+1} - 1}{p - 1} \right) \left( \frac{q^{m+1} - 1}{q - 1}\right)}{2^{r}q^{m}p^{n}}$}, for $r \geq 3$, $p\neq q$.
\end{itemize}
    
\end{prop}

\begin{proof}

First, we simplify \(2 \cdot (p^n \cdot 2^{m-2})\): $2 \cdot p^n \cdot 2^{m-2} = 2^{1 + m - 2} \cdot p^n = 2^{m-1} \cdot p^n$. Thus, the original expression becomes: $\tau(2^{m-1} \cdot p^n) + \sigma(p^n \cdot 2^{m-2})$. Since \(\tau\) is multiplicative and \(2\) and \(p\) are distinct primes, we have: $\tau(2^{m-1} \cdot p^n) = \tau(2^{m-1}) \cdot \tau(p^n)$. We know that for a prime \(q\) raised to a power \(k\), \(\tau(q^k) = k + 1\). Therefore: $\tau(2^{m-1}) = (m - 1) + 1 = m$ and $\tau(p^n) = n + 1$. Therefore, $\tau(2^{m-1} \cdot p^n) = m \cdot (n + 1)$.
Since \(\sigma\) is multiplicative and \(p\) and \(2\) are distinct primes, we have:

\begin{align*}
\sigma(p^n \cdot 2^{m-2}) &= \sigma(p^n) \cdot \sigma(2^{m-2})\\
                          &=  \left( \frac{p^{n+1} - 1}{p - 1} \right) \cdot \frac{2^{(m-2)+1} - 1}{2 - 1} \\
                          &= \left( \frac{p^{n+1} - 1}{p - 1} \right) \cdot   2^{m-1} - 1.\\
\end{align*}

Then it follows
\[
|L(C_{p^{n}} \rtimes Q_{2^{m}})| = m(n + 1) + \left(\frac{p^{n+1} - 1}{p - 1}\right)(2^{m-1} - 1).
\]

Now for $G = C_{q^{m}} \rtimes (C_{p^{n}} \rtimes C_4)$. Note that the number $2 \cdot p^n \cdot q^m$ has decomposition: $2^1 \cdot p^n \cdot q^m$, therefore $\tau = (1+1)(n+1)(m+1) = 2(n+1)(m+1)$. Now just do:

\begin{align*}
    \sigma(p^n \cdot q^m) = \sigma(p^n) \cdot \sigma(q^m) = \left( \frac{p^{n+1} - 1}{p - 1} \right) \left( \frac{q^{m+1} - 1}{q - 1} \right)
\end{align*}

Therefore, 

$$
|L(C_{q^{m}} \rtimes (C_{p^{n}} \rtimes C_4))| =
2(n+1)(m+1) + \left( \frac{p^{n+1} - 1}{p - 1} \right) \left( \frac{q^{m+1} - 1}{q - 1} \right).
$$

For $\beta (C_{q^m} \rtimes (C_{p^{n}} \rtimes Q_{2^r}))$ it suffices to proceed in a manner analogous to the previous cases.
\end{proof}

\begin{cor}
  \begin{itemize}[label=$\circ$]
Considering the limits of the proposition \ref{prop 5.2} we have
  
  \item{$ \lim_{n\longrightarrow \infty}\beta(C_{p^{n}} \rtimes Q_{2^{m}}) = \frac{p}{p-1} (\frac{1}{2} -\frac{1}{2^{m}}) \,  \text{for m} \geq 3$};
    \item{$\lim_{n\longrightarrow \infty}\beta (C_{q^m} \rtimes (C_{p^{n}} \rtimes C_4)) = \frac{p}{4(p - 1)(q - 1)} \cdot \frac{q^{m+1} - 1}{q^m}$};
    \item{$\lim_{n\longrightarrow \infty} \beta (C_{q^m} \rtimes (C_{p^{n}} \rtimes Q_{2^r})) =  \left( \frac{2^{r-1} - 1}{2^r} \right) \cdot \frac{p}{p - 1} \cdot \frac{q^{m+1}}{q^m(q - 1)}
$, for $r \geq 3$};

  \item{$\lim_{m\longrightarrow \infty}\beta(C_{p^{n}} \rtimes Q_{2^{m}}) = \frac{1}{2p^{n}} \left( \frac{p^{n+1}-1}{p-1} \right)$};
    \item{$\lim_{m\longrightarrow \infty}\beta (C_{q^m} \rtimes (C_{p^{n}} \rtimes C_4)) = \frac{q}{4(p - 1)(q - 1)} \cdot \frac{p^{n+1} - 1}{p^n}$}.
    \item{$\lim_{m\longrightarrow \infty} \beta (C_{q^m} \rtimes (C_{p^{n}} \rtimes Q_{2^r})) = \frac{(2^{r-1} - 1)(p^{n+1} - 1)q}{2^r p^n (p - 1)(q - 1)}$, for $r \geq 3$};
    \item{$\lim_{r\longrightarrow \infty} \beta (C_{q^m} \rtimes (C_{p^{n}} \rtimes Q_{2^r})) = \frac{\left( \frac{p^{n+1} - 1}{p - 1} \right) \left( \frac{q^{m+1} - 1}{q - 1} \right)}{2 q^{m} p^{n}}$}.
\end{itemize}
\end{cor}

\begin{ex}\label{Ex 5.4}
Using the mathematical expressions above, it is possible to obtain the following expressions

\begin{table}[h!]
\centering
%\resizebox{\textwidth}{!}{%
  \begin{tabular}{lccccc}
    \hline
    $G$ & $\beta(G)$ & $\alpha(G)$ & $cdeg(G)$ & $ndeg(G)$ \\
    \hline
    $C_{p^n} \rtimes C_{4}$ & $\frac{2(n+1) + \frac{p^{n+1} -1}{p-1}}{4p^{n}}$ & $\frac{p^{n} + 2(n+1)}{4p^n}$ & $\frac{p^{n} + 2(n+1)}{2(n+1) + \frac{p^{n+1} -1}{p-1}}$ & $\frac{2n+3}{2(n+1) + \frac{p^{n+1} -1}{p-1}}$ \\
    $C_{p^n} \rtimes Q_{8}$ & $\frac{3 \left[(n+1) +  \frac{p^{n+1} -1}{p-1} \right]}{8p^{n}}$ & $\frac{2p^{n} + 3(n+1)}{8p^n}$ & $\frac{2p^{n} + 3(n+1)}{3 \left[(n+1) +  \frac{p^{n+1} -1}{p-1} \right]}$ & $\frac{3n+6}{3 \left[(n+1) +  \frac{p^{n+1} -1}{p-1} \right]}$ \\
    \hline
  \end{tabular}
\end{table}
\end{ex}

For the groups of Example \ref{Ex 5.4} we have when $n \longrightarrow \infty$

\begin{table}[h!]
\centering
%\resizebox{\textwidth}{!}{%
  \begin{tabular}{lccccc}
    \hline
    $G$ & $\beta(G)_{\infty}$ & $\alpha(G)_{\infty}$ & $cdeg(G)_{\infty}$ & $ndeg(G)_{\infty}$ \\
    \hline
    $C_{p^{n}} \rtimes C_{4}$ & $\frac{1}{4}$ & $\frac{1}{4}$ & $1$ & $0$ \\
    $C_{p^{n}} \rtimes Q_{8}$ & $\frac{3}{8}$ & $\frac{1}{4}$ & $\frac{2}{3}$ & $0$ \\
    \hline
  \end{tabular}%
%}
%\caption{Parâmetros de alguns produtos semidiretos}
\label{tab:ca}
\end{table}  

\subsection{Density Results}

\quad \, In \cite{Marius} the question of the density of values of $\mathrm{cdeg}(G)$ on $[0,1]$ was raised. This question was fully resolved by Lazorec \cite{lazorec2024elementordersextraspecialgroups}; see also \cite{SHARMA_2025} for a related approach based on direct products. The next result offers an alternative route: we identify a new class of groups (e.g., extensions of the type $C_{p^2}\rtimes C_p$) whose products again provide a dense set of degrees of cyclicity, thus constituting an independent proof that complements the existing literature.

\begin{lem}
   Let $G = C_{p^2} \rtimes C_{p}$, then
\begin{equation*}
    cdeg(G) = \frac{p+1}{p+2}.
\end{equation*}    
\end{lem}

\begin{proof}
Based on Lemma 3.2.1 in \cite{Marius} we have that $|L(C_{p^2} \rtimes C_{p})| = 2p+4$. Using Theorem 2 of \cite{Garonzi_2018} we have that $c(C_{p^2} \rtimes C_{p}) = 2p+2$. Thus, we have

$$
cdeg(C_{p^2} \rtimes C_{p}) = \frac{2p+2}{2p+4} = \frac{2(p+1)}{2(p+2)} = \frac{p+1}{p+2}.
$$

\end{proof}

\begin{theorem}
Let $C_{p^2} \rtimes C_{p}$, where $p \geq 3$ with p prime. Then the set $ cdeg(\bigotimes_{i \in I}C_{p^2_{i}} \rtimes C_{p_{i}}),|I| < \infty, \, \text{and} \, \, p_{i} \, \text{is the ith prime number}, \forall i \in I \}$ is dense in $(0, 1]$.
\end{theorem}

\begin{proof}
According to \cite{Lazorec}'s Lemma 4.1 it is sufficient to take $(x_n)_{n \geq 1}$ to be a sequence of positive real numbers such that $\lim_{n \longrightarrow \infty} x_n = 0$ and $\sum_{n=1}^{\infty} x_{n}$ is divergent to show that the sequence is dense in $[0,\infty)$. For this case let us take $x_{n} = -\ln(cdeg(C_{p^2_{n}} \rtimes C_{p_{n}}))$. Clearly $x_{n}>0$ and $\lim_{n \longrightarrow \infty} x_n = 0$. Now it remains to show that the series $\sum_{n=1}^{\infty} x_{n} = -\sum_{n=1}^{\infty}\ln(cdeg(C_{p^2_{n}} \rtimes C_{p_{n}}))$ is divergent. Indeed:

\begin{align*}
 \sum_{n=1}^{\infty} x_{n} = -\sum_{n=1}^{\infty}\ln(cdeg(C_{p^2_{n}} \rtimes C_{p_{n}}))  &=  -\sum_{n=1}^{\infty}\ln\left(\frac{p_{n}+1}{p_{n}+2}\right)\\ 
 &= \sum_{n=1}^{\infty}\ln\left(\frac{p_{n}+2}{p_{n}+1}\right)\\
 &= \ln\left( \prod_{i=1}^{\infty} \frac{p_{n}+2}{p_{n}+1} \right).
\end{align*}

As $p_{n}$ increases the ratio also increases. Since the logarithm is a monotone and increasing function, it follows that the expression tends to infinity slowly, and therefore diverges. Therefore, according to Lemma 4.1 of \cite{Lazorec}, it follows that

\[
\overline{
\left\{
-\ln \prod_{i \in I} cdeg(C_{p^2_{n}} \rtimes C_{p_{n}}) \;\middle|\; I \subset \mathbb{N} \setminus \{0\},\, |I| < \infty \text{ and } p_i \text{ is the } i\text{th prime number},\, \forall i \in I
\right\}
} = [0, \infty).
\]

Since the exponential function is continuous and the above relation expresses an equality between the closures of two sets of $\mathbb{R}$, we obtain
\[
\overline{
\left\{
\prod_{i \in I} cdeg(C_{p^2_{n}} \rtimes C_{p_{n}}) \;\middle|\; I \subset \mathbb{N} \setminus \{0\},\, |I| < \infty \text{ and } p_i \text{ is the } i\text{th prime number},\, \forall i \in I
\right\}
} = (0, 1].
\]

The conclusion follows by applying the multiplicativity of $cdeg$  for the left-hand side of the above equality. 

\end{proof}

Note that this class of groups is not the only one that has this property. The groups $C_{p^2} \times C_{p}$ also have this property. To observe it more clearly, it is necessary to observe the behavior of cdeg for $C_{p^2} \times C_{p}$, considering Corollary 3.1.2 in \cite{Marius} taking $\alpha_1 = 2$ and $\alpha_2 = 1$ the result follows.

% $$
% cdeg(C_{p^2} \rtimes C_{p}) = \frac{2p+2}{2p+4} = \frac{2(p+1)}{2(p+2)} = \frac{p+1}{p+2}.
% $$

\section{Conclusion}

\quad \, We exploit the identity of Theorem \ref{T1} which serves as a unifying lens to quickly organize and deduce asymptotic values and bounds of $\alpha$, $\beta$ and $\mathrm{cdeg}$ in classical families of finite groups. From it, we obtain explicit formulas and bound behaviors for $D_{2^n}$, $Q_{2^n}$ and $SD_{2^n}$, as well as for 2-power abelian classes and for Hamiltonian groups of type $Q_8\times C_2^n\times A$, highlighting the relation $\alpha(Q_8\times C_2^n\times A)=\tfrac{5}{8}\,\alpha(A)$. In the dicyclic case, we recover the subgroup count using $\lvert L(\mathrm{Dic}_n)\rvert=\sigma(n)+\tau(2n)$, which directly yields $\beta(\mathrm{Dic}_n)$ and allows us to illustrate the values with numerical examples.

Regarding the question of the density of values of $\mathrm{cdeg}(G)$ in $[0,1]$, proposed in \cite{Marius}, we note that it was completely resolved by Lazorec \cite{lazorec2024elementordersextraspecialgroups} (see also \cite{SHARMA_2025}). In parallel, we show that certain elementary semidirects (e.g., groups of the type $C_{p^2}\rtimes C_p$) generate, via direct products, a family of explicit rational values whose multiplication again produces a dense set; this alternative approach is conceptually simple and complements the existing overview. To facilitate reference, we have compiled a summary table of the main expressions, limits, and internal references, which has been moved to the appendix (see Table~\ref{tab:summary}).

In summary, this work consolidates and streamlines calculations of $\alpha$, $\beta$, and $\mathrm{cdeg}$ in several families, offers a unified view through the identity $\mathrm{cdeg}=\alpha/\beta$, and presents an alternative construction for the density of $\mathrm{cdeg}$, contributing to a clearer picture of its distribution in finite groups.

% \subsection{Acknowledgments}

% \quad \, 

\newpage
%\section{Considerações Finais}

% \bibliographystyle{plain}  % Estilo de citação, 'plain' é um exemplo
% \bibliography{bibliography}  % Nome do arquivo .bib sem a extensão
% (opcional no preâmbulo)
% \usepackage[T1]{fontenc}
% \usepackage[utf8]{inputenc}
% \usepackage[brazil]{babel}

% ... no final do documento:

\newpage
\section{Appendix}

\newsavebox{\tempbox}
\setbox\tempbox=\vbox{\section{Appendix}}  % Measure the height of a dummy section (reuse your actual section above)
\rotatebox{90}{%
  \begin{minipage}[c][\textwidth][c]{\dimexpr \textheight - \ht\tempbox - \dp\tempbox - \baselineskip - \parskip}
    \centering
    \captionof{table}{Summary of properties for various group families (Translated and reoriented)}
    \label{tab:summary}
    \resizebox{1\textwidth}{!}{%  % Adjust to 0.9 or less if needed; now uses \textwidth as effective height
      \small
      \begin{tabular}{@{}lcccccccccc@{}}
        \toprule
        \textbf{Property} & \textbf{$D_{2^n}$} & \textbf{$Q_{2^n}$} & \textbf{$SD_{2^n}$} & \textbf{$C_{2^{n+1}}\times C_2$} & \textbf{$C_{2^{n}p}\times C_2$} & \textbf{$Q_8\times C_2^{n}\times A$} & \textbf{$\mathrm{Dic}_n$} & \textbf{$C_{p^n}\rtimes C_4$} & \textbf{$C_{p^n}\rtimes Q_8$} & \textbf{$C_{p^2}\rtimes C_p$} \\
        & (dihedral) & (gen. quat.) & (semidihedral) & & ($p\ge 3$ prime) & (hamiltonian) & (dicyclic) & ($p$ prime) & ($p$ prime) & ($p\ge 3$ prime) \\
        \midrule
        \textbf{Parameter} & $n\ge 2$ & $n\ge 3$ & $n\ge 4$ & $n\ge 1$ & $n\ge 1$ & \begin{tabular}{@{}c@{}}$n\ge 0$, \\ $A$ odd abelian\end{tabular} & $|G|=4n$ & $n\ge 1$ & $n\ge 1$ & $p\ge 3$ \\
        \textbf{$|G|$} & $2^{n}$ & $2^{n}$ & $2^{n}$ & $2^{n+2}$ & $2^{n+1}p$ & $2^{n+3}|A|$ & $4n$ & $4p^{n}$ & $8p^{n}$ & $p^{3}$ \\
        \textbf{$|L(G)|$} & $2^{n}+n-1$ & $2^{n-1}+n-1$ & $3\cdot 2^{n-2}+n-1$ & $3n+5$ & $6n+4$ & $b_{n,2}\cdot |L(A)|$ & $\tau(2n)+\sigma(n)$ & $2(n+1)+\frac{p^{n+1}-1}{p-1}$ & $3\left[(n+1)+\frac{p^{n+1}-1}{p-1}\right]$ & $2p+4$ \\
        \textbf{$c(G)$} & $2^{n-1}+n$ & $2^{n-2}+n$ & $3\cdot 2^{n-3}+n$ & $2(n+2)$ & $4(n+1)$ & $5\cdot 2^{n}\, c(A)$ & $-$ & $\frac{p^{n}+2(n+1)}{2}$ & $\frac{2p^{n}+3(n+1)}{2}$ & $2p+2$ \\
        \textbf{$\alpha(G)$} & $\dfrac{2^{n-1}+n}{2^{n}}$ & $\dfrac{2^{n-2}+n}{2^{n}}$ & $\dfrac{3\cdot 2^{n-3}+n}{2^{n}}$ & $\dfrac{n+2}{2^{n+1}}$ & $\dfrac{2(n+1)}{2^{n}p}$ & $\dfrac{5}{8}\,\alpha(A)$ & $-$ & $\dfrac{p^{n}+2(n+1)}{4p^{n}}$ & $\dfrac{2p^{n}+3(n+1)}{8p^{n}}$ & $\dfrac{2p+2}{p^{3}}$ \\
        \textbf{$\beta(G)$} & $\dfrac{2^{n}+n-1}{2^{n}}$ & $\dfrac{2^{n-1}+n-1}{2^{n}}$ & $\dfrac{3\cdot 2^{n-2}+n-1}{2^{n}}$ & $\dfrac{3n+5}{2^{n+2}}$ & $\dfrac{3n+2}{2^{n}p}$ & $\beta(A)\cdot \dfrac{b_{n,2}}{2^{n+3}}$ & $\dfrac{\tau(2n)+\sigma(n)}{4n}$ & $\dfrac{2(n+1)+\frac{p^{n+1}-1}{p-1}}{4p^{n}}$ & $\dfrac{3\left[(n+1)+\frac{p^{n+1}-1}{p-1}\right]}{8p^{n}}$ & $\dfrac{2p+4}{p^{3}}$ \\
        \textbf{$\mathrm{cdeg}(G)$} & $\dfrac{2^{n-1}+n}{2^{n}+n-1}$ & $\dfrac{2^{n-2}+n}{2^{n-1}+n-1}$ & $\dfrac{3\cdot 2^{n-3}+n}{3\cdot 2^{n-2}+n-1}$ & $\dfrac{2(n+2)}{3n+5}$ & $\dfrac{2(n+1)}{3n+2}$ & $\dfrac{5\cdot 2^{n}}{b_{n,2}}\ \mathrm{cdeg}(A)$ & $-$ & $\dfrac{p^{n}+2(n+1)}{2\left[(n+1)+\frac{p^{n+1}-1}{p-1}\right]}$ & $\dfrac{2p^{n}+3(n+1)}{3\left[(n+1)+\frac{p^{n+1}-1}{p-1}\right]}$ & $\dfrac{p+1}{p+2}$ \\
        \textbf{Limit} & \begin{tabular}{@{}c@{}}$\alpha\to \tfrac12$\\$\beta\to 1$\\$\mathrm{cdeg}\to \tfrac12$\end{tabular} & \begin{tabular}{@{}c@{}}$\alpha\to \tfrac14$\\$\beta\to \tfrac12$\\$\mathrm{cdeg}\to \tfrac12$\end{tabular} & \begin{tabular}{@{}c@{}}$\alpha\to \tfrac38$\\$\beta\to \tfrac34$\\$\mathrm{cdeg}\to \tfrac12$\end{tabular} & \begin{tabular}{@{}c@{}}$\alpha\to 0$\\$\beta\to 0$\\$\mathrm{cdeg}\to \tfrac23$\end{tabular} & \begin{tabular}{@{}c@{}}$\alpha\to 0$\\$\beta\to 0$\\$\mathrm{cdeg}\to \tfrac23$\end{tabular} & \begin{tabular}{@{}c@{}}See Cor. 4.9 \\ e.g. $A=C_p^m$: \\ $\alpha\to \tfrac{5}{8(p-1)}$\end{tabular} & $-$ & \begin{tabular}{@{}c@{}}$\alpha\to \tfrac14$\\$\beta\to \tfrac{p}{4(p-1)}$\\$\mathrm{cdeg}\to \tfrac{p-1}{p}$\end{tabular} & \begin{tabular}{@{}c@{}}$\alpha\to \tfrac14$\\$\beta\to \tfrac{3p}{8(p-1)}$\\$\mathrm{cdeg}\to \tfrac{2(p-1)}{3p}$\end{tabular} & \begin{tabular}{@{}c@{}}$p\to\infty:$ \\ $\alpha,\beta\to 0$\\$\mathrm{cdeg}\to 1$\end{tabular} \\
        \textbf{Source} & \begin{tabular}{@{}c@{}}Prop. 3.2, 3.4;\\Cor. 3.3, 3.5\end{tabular} & \begin{tabular}{@{}c@{}}Prop. 3.2, 3.4;\\Cor. 3.3, 3.5\end{tabular} & \begin{tabular}{@{}c@{}}Prop. 3.2, 3.4;\\Cor. 3.3, 3.5\end{tabular} & \begin{tabular}{@{}c@{}}Prop. 4.1, 4.5;\\Cor. 4.6\end{tabular} & \begin{tabular}{@{}c@{}}Prop. 4.3, 4.5;\\Cor. 4.4, 4.6\end{tabular} & Cor. 4.9 & Thm. 5.1 & \begin{tabular}{@{}c@{}}Ex. 5.4 \\ (depends \\ on Prop. 5.2)\end{tabular} & \begin{tabular}{@{}c@{}}Ex. 5.4 \\ (depends \\ on Prop. 5.2)\end{tabular} & \begin{tabular}{@{}c@{}}Lemma 5.5 \\ (counts used \\ in the proof)\end{tabular}\\
        \bottomrule
      \end{tabular}
    }
  \end{minipage}}

\end{document}